\documentclass[11pt,twoside]{article}
\usepackage{amsthm}
\usepackage{amssymb}
\usepackage{amsmath}
\usepackage{mathrsfs}
\usepackage{txfonts}

\allowdisplaybreaks

\def\rr{{\mathbb R}}
\def\rn{{{\rr}^n}}

\def\cn{{\mathbb N}}

\def\fz{\infty}
\def\az{\alpha}
\def\supp{{\mathop\mathrm{\,supp\,}}}

\def\loc{{\mathop\mathrm{\,loc\,}}}
\def\lip{{\mathop\mathrm{\,Lip\,}}}
\def\lz{\lambda}
\def\dz{\delta}

\def\ez{\epsilon}

\def\bz{\beta}

\def\gz{{\gamma}}

\def\pa{\partial}

\def\wz{\widetilde}

\def\hs{\hspace{0.3cm}}

\def\com{\complement}
\def\r{\right}
\def\lf{\left}

\newtheorem{thm}{Theorem}[section]
\newtheorem{lem}{Lemma}[section]
\newtheorem{prop}{Proposition}[section]

\numberwithin{equation}{section}

\begin{document}
\arraycolsep=1pt
\title{\Large\bf Lipschitz continuity of solutions of Poisson
equations\\
in metric measure spaces\footnotetext{\hspace{-0.35cm} 2000 {\it
Mathematics Subject Classification}. 31C25; 31B05; 35B05; 35B45
\endgraf{\it Key words and phrases. {\rm Lipschitz regularity; Poincar\'e inequality;
Newtonian space; heat kernel; Poisson equation}}
\endgraf Renjin Jiang was partially supported by the Academy of Finland grants 120972 and
131477.}}

\author{Renjin Jiang}
\date{ }
\maketitle
\begin{center}
\begin{minipage}{13.5cm}\small
{\noindent{\bf Abstract.} Let $(X,d)$ be a pathwise connected metric space equipped with
an Ahlfors $Q$-regular measure $\mu$, $Q\in[1,\infty)$. Suppose that $(X,d,\mu)$ supports a
$2$-Poincar\'e inequality and a Sobolev-Poincar\'e type inequality
for the corresponding ``Gaussian measure". The author uses the heat equation to study the
Lipschitz regularity of solutions of the Poisson equation $\Delta u=f$, where $f\in L^p_\loc$.
When $p>Q$, the local Lipschitz continuity of $u$ is established. }
\end{minipage}
\end{center}
\vspace{0.0cm}

\section{Introduction}
\hskip\parindent Let $(X,d)$ be a pathwise connected, proper metric
measure space, where proper means: each closed ball in $X$ is
compact. Given a domain $\Omega \subseteq X$
and $u$ a measurable function on $\Omega$, a non-negative Borel
function $g$ is called an upper gradient of $u$ on $\Omega$, if
\begin{equation*}
 |u(x)-u(y)|\le \int_{\gz} g\,ds
\end{equation*}
for all $x,\,y\in\Omega$ and each rectifiable curve $\gz:\,[0,l]\to
\Omega$ that joins $x$ and $y$. Further, a metric measure space $(X,d,\mu)$ is
said to support a (weak) $p$-Poincar\'e inequality, if there exist
$C_P>0$ and $\lz\ge1$ such that for every ball $B(x,r)\subseteq X$ and
for each continuous function $u$ and every upper gradient $g$ of $u$
on $B(x,\lz r)$,
\begin{equation*}
\fint_{B(x,r)}|u(y)-u_B|\,d\mu(y)\le
C_Pr\lf(\fint_{B(x,\lz
r)}g(y)^p\,d\mu(y)\r)^{1/p},
\end{equation*}
where and in what follows, for each ball $B\subset X$, $u_B=\fint_B u\,d\mu=\mu(B)^{-1}\int_{B}u\,d\mu$;
see \cite{hek} for details.

By using the upper gradient, Shanmugalingam \cite{sh} introduced
the first-order Sobolev spaces on $X$, i.e., the Newtonian (Sobolev) space
$N^{1,p}(X,\mu)$.  For convenience, we denote the local
Newtonian spaces and Newtonian spaces with zero boundary values
by $N_{\loc}^{1,p}$ and $N_{0}^{1,p}$, respectively (see Section 2 for
details). We note that it was proved in \cite{sh} that the Newtonian (Sobolev) spaces $N^{1,p}(X,\mu)$
coincide with the Sobolev spaces introduced by Cheeger \cite{ch}
for $p>1$. From \cite{ch}, for each $u\in N^{1,p}(X,\mu)$,
we can assign a differential $Du$, which is called Cheeger derivative of $u$ following
\cite{krs}; see Subsection 2.1 below. Notice that for Lipschitz functions $u$,
the inner product $Du\cdot Du$ is comparable to the square of $\lip u$, where
\begin{equation*}
\lip u(x)=\limsup_{r\to 0}\sup_{d(x,y)\le r}\frac{|u(x)-u(y)|}{r}.
\end{equation*}

Having the above tools, the Lipschitz regularity of harmonic
functions in $X$ is then considered in \cite{krs}.
Let us first recall some
notions. Let $\mu$ be a $Q$-regular measure on $X$ for some $Q\ge 1$,
i.e., $\mu$ is Borel-regular and there exist constants $Q\ge 1$ and
$C_Q\ge 1$ such that for every $x\in X$ and all $r>0$,
$$C_Q^{-1}r^Q\le \mu(B(x,r))\le C_Qr^Q.$$
Let $\Omega\subseteq X$ be a domain. A function $u\in N_{\loc}^{1,2}(X)$ is called
Cheeger-harmonic in $\Omega$, if for all Lipschitz functions $\phi$
with compact support in $\Omega$,
\begin{equation*}
  \int_\Omega Du(x)\cdot D\phi(x)\,d\mu(x)=0.
\end{equation*}

The following theorem was established in \cite{krs}.
\begin{thm}
Let  $Q > 1$ and suppose that $(X,d,\mu)$ supports a
2-Poincar\'e inequality. Furthermore, assume that there
exist constants $C > 0$ and $t_0>0$ such that for each $0<t<t_0$ and
every $g\in N^{1,2}(X)$,
\begin{eqnarray}\label{1.1}
\int_X g(y)^2 p(t,x,y)\,d\mu(y)&&\le (2t+Ct^2)\int_X |Dg(y)|^2 p(t,x,y)\,d\mu(y)\nonumber\\
&&\hs+\lf(\int_X g(y)p(t,x,y)\,d\mu(y)\r)^2
\end{eqnarray}
for almost every $x\in X$. If $u$ is Cheeger-harmonic in $\Omega$, where $\Omega\subset X$
is a domain, then $u$ is locally Lipschitz continuous in $\Omega$.
\end{thm}
Above, $p(t,x,y)$ refers to the heat kernel associated to the Dirichlet
form $\int_X Df(x)\cdot Dg(x)\,d\mu(x)$; see Subsection 2.2 below.

It is well known that \eqref{1.1} can be deduced from the logarithmic Sobolev inequality
\begin{eqnarray}\label{1.2}
&&\int_X f(x)^2\log\lf(\frac{f(x)^2}{\|f\|^2_{L^2({X,p(t,x_0,x)\,d\mu})}}\r)p(t,x_0,x)\,d\mu(x)\nonumber\\
&&\hs\hs\hs\le (4t+2Ct^2)\int_X|\nabla f(x)|^2p(t,x_0,x)\,d\mu(x);
\end{eqnarray}
see, for example, \cite{bak1}.

We remark that the authors in \cite{krs} gave several examples
to show that: (i) in the abstract settings,
harmonic functions may not be smooth and local Lipschitz continuity may be
the best possible regularity; (ii) doubling of $\mu$ is not enough
to guarantee the local Lipschitz continuity of harmonic functions and it
is natural to consider an Ahlfors $Q$-regular measure;
(iii) even when the Poincar\'e inequality and Ahlfors $Q$-regularity hold,
harmonic functions may still not be locally Lipschitz continuous; hence
a Sobolev-Poincar\'e inequality \eqref{1.1} is needed.

Inspired by \cite{krs}, in this paper, we work on the Lipschitz
regularity of solutions of the Poisson equations in metric spaces.
Let $\Omega\subseteq X$ be a domain.
A Sobolev function $u\in N^{1,2}_{\loc}(X)$ is called a solution of
the equation $\Delta u=f$ in $\Omega$, if
\begin{equation}\label{1.3}
-\int_\Omega Du(x)\cdot D\phi(x)\,d\mu(x)=\int_\Omega f(x)\phi(x)\,d\mu(x),
\ \ \forall \phi\in N_0^{1,2}(\Omega).
\end{equation}

\begin{thm}\label{t1}
Let $Q \ge 1$ and suppose that $(X,d,\mu)$ supports a
2-Poincar\'e inequality and that \eqref{1.1} holds.
Let $u\in N^{1,2}_{\loc}(X)$ satisfy $\Delta u=f$ in $\Omega,$
where $\Omega\subseteq X$ is a domain and $f\in L^p(\Omega)$.
If $p>Q$, then $u$ is locally Lipschitz continuous in $\Omega$.
\end{thm}

Notice that in our abstract cases, we can only define the first
order derivative (Cheeger derivative), and the space $X$ does not
have any geometric structure. Thus many classical methods are not available
for Theorem 1.2. In this paper, we inherit the method investigated
in \cite{ck} and used by Koskela et al \cite{krs} to study the Lipschitz
regularity of harmonic functions in metric measure spaces.
The method involves the abstract theory of Dirichlet forms
and the heat equation. It is worth pointing out that the method used
in \cite{ck,krs} can not be directly adapted to our setting:
we need to modify the definition of a crucial functional  $J(t)$;
see \cite[1.1.5]{ck}, \cite[p.167]{krs} and \eqref{3.1} below.

The paper is organized as follows. In Section 2, we give some basic notation
and notions for Newtonian spaces, the Cheeger derivative and Dirichlet forms.
Several auxiliary results are also given. Section 3 is devoted to the proof of
Theorem 1.2. In Section 4, we discuss some situations where the Sobolev-Poincar\'e
inequality \eqref{1.1} can be verified.

Finally, we make some conventions. Throughout the paper, we denote
by $C$ a positive constant which is independent of the main
parameters, but which may vary from line to line. We also use $C_{\gz,\bz, \cdots}$
to denote a positive constant depending on the indicated parameters
$\gz,\bz,\cdots$.

\section{Preliminaries}
\hskip\parindent In this section, we give some basic notation
and notions and several auxiliary results.
\subsection{Cheeger derivative in metric spaces}
\hskip\parindent Let us first recall the definition of Newtonian spaces on $X$
following \cite{sh}. Notice that the measure $\mu$ is only required
to be doubling in \cite{sh} and Ahlfors $Q$-regular measures
are always doubling measures.

The Newtonian space $N^{1,p}(X)$ is defined to be the space of all $p$-integrable
(equivalence classes of) functions for which there exists a $p$-integrable upper
gradient. If $u\in N^{1,p}(X)$, then we define its pseudonorm by
$$\|u\|_{N^{1,p}(X)}:=\|u\|_{L^p(X)}+\inf_{g}\|g\|_{L^p(X)}.$$
where the infimum is taken over all upper gradients of $u$. Further, if
$(X,d,\mu)$ supports the $p$-Poincar\'e inequality, then it is proved in
\cite{sh} that the set of all Lipschitz functions are dense in $N^{1,p}(X)$.
It is then natural to define the Newtonian spaces on open subsets and
local Newtonian spaces $N^{1,p}_{\loc}(X)$. The Sobolev spaces with zero
boundary values on metric spaces were studied in \cite{kkm}. For a domain
$\Omega\subset X$, following \cite{kkm}, we define the Newtonian space
with zero boundary values $N^{1,p}_0(X)$ to be the space of $u\in N^{1,p}(X)$
for which $u\chi_{X\setminus \Omega}$ vanishes $p$-quasi everywhere.
Recall that a property holds $p$-quasi everywhere, if it holds except of
a set of $p$-capacity zero.

Cheeger \cite{ch} introduced Sobolev spaces in metric spaces by using
upper gradients in a different way, but it was proved in \cite{sh}
that the Sobolev spaces in \cite{ch} coincide with the corresponding
Newtonian spaces for $p > 1$. The following theorem established in \cite{ch} provides us with
a differential structure on metric spaces.

\begin{thm}\label{t2.1}
Assume that $(X,d,\mu)$ supports a weak $p$-Poincar\'e inequality for some
$p>1$ and that $\mu$ is doubling. Then there exists $N > 0$, depending only on the doubling
constant and the constants in the Poincar\'e inequality, such that the following holds.
There exists a countable collection of measurable sets $U_\az$, $\mu(U_\az)>0$
for all $\az$, and Lipschitz functions $X^\az_1,\cdots, X^\az_{k(\az)}:\, U_\az\to \rr$,
with $1\le k(\az)\le N$ such that
$\mu(X \setminus \cup_{\az=1}^\fz U_\az)=0,$
and for all $\az$ and $X^\az_1,\cdots, X^\az_{k(\az)}$ the following holds:
for $f: X\to R$ Lipschitz, there exist $V_\az(f)\subseteq U_\az$ such that
$\mu(U_\az\setminus V_\az(f))=0$, and Borel functions $b_1^\az(x,f),\cdots,b^\az_{k(\az)}(x,f)$
of class $L^\fz$ such that if $x\in V_\az$, then
$$\lip(f-a_1 X_1^\az-\cdots-a_{k(\az)}X^\az_{k(\az)})(x)=0$$
if and only if $(a_1,\cdots,a_{k(\az)})=(b^\az_1(x,f),\cdots,b^\az_{k(\az)}(x,f))$.
Moreover, for almost every $x\in U_{\az_1}\cap U_{\az_2}$, the ``coordinate functions"
$X_i^{\az_2}$ are linear combinations of the $X_i^{\az_1}$'s.
\end{thm}
By Theorem \ref{t2.1}, for each Lipschitz function $u$ we can assign a derivative $Du$,
which maps $V_\az\subset X$ into $\rr^{k(\az)}$, where $\az\in\cn$ and $V_\az$, $k(\az)$ are
as above. Moreover, the differential operator $D$,
and hence the Laplacian operator $\Delta$, depend on the charts chosen.

Cheeger further showed that
the differential operator $D$ can be extended to all functions in the corresponding
Sobolev spaces, which coincide with the Newtonian spaces $N^{1,p}(X)$.
A useful fact is that the Cheeger derivative satisfies the Leibniz rule, i.e.,
$$D(uv)(x)=u(x)Dv(x)+v(x)Du(x).$$
Moreover, the Euclidean norm $|Du|$ of $Du$ is comparable to $\lip u$.

We now introduce several useful inequalities.
Recall that $(X,\mu)$ is an Ahlfors $Q$-regular space and supports
a weak $2$-Poincar\'e inequality. Then there exists a positive
constant $C$, only depending on $Q, C_P$ and $C_Q$, such that for all $u\in N^{1,2}_0(B(x,r))$
\begin{eqnarray}\label{2.1}
\lf(\fint_{B(x,r)}|u(y)|^s\,d\mu(y)\r)^{1/s}\le C r \lf(\fint_{B(x,r)}|Du(y)|^2\,d\mu(y)\r)^{1/2},
\end{eqnarray}
where $s=\frac{2Q}{Q-2}$ if $Q>2$, and $s>2$ is arbitrary if $Q\le 2$; see \cite[(k)]{bm} and
also \cite{bm93,hak,sal}.
Notice that a (2,2)-Poincar\'e inequality is required
in \cite{bm}, but this is equivalent to the 2-Poincar\'e inequality under our
assumptions; see \cite{hak95,hak}.

\begin{lem}\label{l2.1}
Let $p>Q$. There exists $C>0$ such that for all $u\in N^{1,2}_{\loc}(X)$
that satisfy $\Delta u=f$ in $B(x_0,2r)$, where $f\in L^p_\loc(X)$
and $B(x_0,2r)\subset\subset X$ it holds
\begin{eqnarray}\label{2.2}
\sup_{B(x_0,r)}|u|\le C[r^{-Q/2}\|u\|_{L^2(B(x_0,2r))}
+r^{2-Q/p}\|f\|_{L^p(B(x_0, 2r))}].
\end{eqnarray}
\end{lem}
\begin{proof}
When $Q\ge 2$, by \cite[p.131]{bm}, there exists
$\wz u \in N_0^{1,2}(B(x_0,2r))$ such that $\Delta \wz u=f$ in $B(x_0,2r)$.
Then from \cite[Theorem 4.1]{bm}, we deduce that
\begin{equation}\label{2.3}
\sup_{B(x_0,2r)}|\wz u|\le Cr^2\mu(B(x_0,r))^{-1/p}\|f\|_{L^p(B(x_0,2r))}.
\end{equation}
Now the fact that $u-\wz u$ is Cheeger-harmonic in $B(x_0, 2r)$ and
\cite[Theorem 5.4]{bm} imply that
\begin{eqnarray*}
\sup_{B(x_0,r)}|u-\wz u|\le Cr^{-Q/2}\|u-\wz u\|_{L^2(B(x_0,2r))},
\end{eqnarray*}
which together with \eqref{2.3} implies that \eqref{2.2} holds.

When $Q\in[1,2)$, we choose a Lipschitz function $\phi$ on $X$ such that $\phi\equiv 1$ on
$B(x_0,r)$, $\supp \phi\subseteq B(x_0,2r)$ and $|D\phi|\le C/r$.
Then it is easy to see that $u\phi\in N_0^{1,2}(B(x_0,2r))$.
By \cite[Theorem 5.1]{hak}, the H\"older inequality
and the Young inequality, we obtain that for every $x\in B(x_0,2r)$,
\begin{eqnarray*}
|(u\phi)(x)|^2&&\le Cr^2\fint_{B(x_0,2r)}|D(u\phi)(y)|^2\,d\mu(y)\\
&&=Cr^2\fint_{B(x_0,2r)}[Du(y)\cdot D(u\phi^2)(y)+|u(y)D\phi(y)|^2]\,d\mu(y)\\
&&\le Cr^2\fint_{B(x_0,2r)}\lf[f(y)u(y)\phi(y)^2+\frac{u(y)^2}{r^{2}}\r]\,d\mu(y)\\
&&\le \lf(\sup_{B(x_0,2r)}|u\phi|\r)\frac{Cr^2\|f\|_{L^p(B(x_0,2r))}}{\mu(B(x_0,2r))^{1/p}}
+\frac{C\|u\|^2_{L^2(B(x_0,2r))}}{\mu(B(x_0,2r))}\\
&&\le \frac12\sup_{B(x_0,2r)}|u\phi|^2+\frac{Cr^4\|f\|^2_{L^p(B(x_0,2r))}}{\mu(B(x_0,2r))^{2/p}}
+\frac{C\|u\|^2_{L^2(B(x_0,2r))}}{\mu(B(x_0,2r))}.
\end{eqnarray*}
Hence
\begin{eqnarray*}
\sup_{B(x_0,r)}|u|\le \sup_{B(x_0,2r)}|(u\phi)|&&\le
\frac{Cr^2\|f\|_{L^p(B(x_0,2r))}}{\mu(B(x_0,2r))^{1/p}}
+\frac{C\|u\|_{L^2(B(x_0,2r))}}{\mu(B(x_0,2r))^{1/2}},
\end{eqnarray*}
which implies that \eqref{2.2} holds, and hence completes the proof of Lemma \ref{l2.1}.
\end{proof}

We also need the following Caccioppoli inequality.
\begin{lem}\label{l2.2} Let $p>Q$.
There exists $C>0$ such that for all $0<r<R$ and $u\in N^{1,2}_{\loc}(X)$
that satisfy $\Delta u=f$ in $B(x_0,2R)$ with $f\in L^p_\loc(X)$ and $B(x_0,2R)\subset\subset X$
it holds
\begin{equation}\label{2.4}
\int_{B(x_0,r)}|Du(x)|^2\,d\mu(x)\le
CR^{2+Q(1-\frac 2p)}\|f\|^2_{L^p(B(x_0,2R))}+\frac{C}{(R-r)^2}\|u\|_{L^2(B(x_0,2R))}^2.
\end{equation}
\end{lem}
\begin{proof} Choose a Lipschitz function $\phi$ such that $\phi=1$ on
$B(x_0,r)$, $\supp \phi\subset B(x_0,R)$ and $|D\phi|\le \frac{C}{R-r}$.
Thus, $u\phi^2\in N_0^{1,2}(B(x_0,R))$. By the Leibniz rule and \eqref{1.3}, we have
\begin{eqnarray*}
\int_{X} Du(x)\cdot D(u\phi^2)(x)\,d\mu(x)&&=\int_{X} \phi(x)^2|Du(x)|^2
+2u(x)\phi(x)Du(x)\cdot D\phi(x)\,d\mu(x)\\
&&=-\int_{X} f(x)u(x)\phi(x)^2\,d\mu(x).
\end{eqnarray*}
Applying the H\"older inequality, the Young inequality and Lemma \ref{l2.1}, we conclude that
\begin{eqnarray*}
&&\int_{X} \phi(x)^2|Du(x)|^2\,d\mu(x)\\
&&\hs\le \int_{X} |f(x)u(x)|\phi(x)^2\,d\mu(x)-\int_{X}2u(x)\phi(x)Du(x)\cdot D\phi(x)\,d\mu(x)\\
&&\hs\le \int_{B(x_0,R)} |f(x)u(x)|\,d\mu(x)
+\frac 12\|\phi|Du|\|_{L^2(X)}^2+8\|u|D\phi|\|_{L^2(X)}^2\\
&&\hs\le CR^{2+Q(1-\frac 2p)}\|f\|^2_{L^p(B(x_0,2R))}
+CR^{Q(\frac 12-\frac 1p)}\|f\|_{L^p(B(x_0,R))}\|u\|_{L^2(B(x_0,2R))}\\
&&\hs\hs+\frac 12\|\phi|Du|\|_{L^2(X)}^2+\frac{C}{(R-r)^2}\|u\|_{L^2(B(x_0,R))}^2\\
&&\hs\le CR^{2+Q(1-\frac 2p)}\|f\|^2_{L^p(B(x_0,2R))}+\frac 12\|\phi|Du|\|_{L^2(X)}^2
+\frac{C}{(R-r)^2}\|u\|_{L^2(B(x_0,2R))}^2,
\end{eqnarray*}
which completes the proof of Lemma \ref{l2.2}.
\end{proof}

\subsection{Dirichlet forms and heat kernels}
\hskip\parindent Having the Sobolev spaces $N^{1,p}(X)$ and the differential operator $D$,
we now turn to the Dirichlet forms on $(X,d,\mu)$.
Define the bilinear form $\mathscr{E}$ by
$$\mathscr{E}(f,g)=\int_X Df(x)\cdot Dg(x)\,d\mu(x)$$
with the domain $D(\mathscr{E})=N^{1,2}(X)$.
Then $\mathscr{E}$ is symmetric and closed. Corresponding to such a form there exists
an infinitesimal generator $A$ which acts on a dense
subspace $D(A)$ of $N^{1,2}(X)$ so that for all $f\in D(A)$ and each
$g\in N^{1,2}(X)$,
$$\int_X g(x)Af(x)\,d\mu(x)=-\mathscr{E}(g,f).$$

Now let us recall several auxiliary results established in \cite{krs}.
\begin{lem}\label{l2.3}
If $u,\,v\in N^{1,2}(X)$, and $\phi\in N^{1,2}(X)$ is a bounded
Lipschitz function, then
\begin{equation*}
\mathscr{E}(\phi,uv)=\mathscr{E}(\phi u,v)+\mathscr{E}(\phi
v,u)-2\int_X \phi Du(x)\cdot Dv(x)\,d\mu(x).
\end{equation*}
Moreover, if $u,\,v\in \mathscr{D}(A)$, then we can unambiguously
define the $L^1$-function $A(uv)$ by setting
\begin{equation*}
A(uv)=uAv+vAu+2Du\cdot Dv.
\end{equation*}
\end{lem}

Also, associated with the Dirichlet form $\mathscr{E}$, there is a
semigroup $\{T_t\}_{t>0}$, acting on $L^2(X)$, with the following
properties (see \cite[Chapter 1]{fot}):

1. $T_t\circ T_s=T_{t+s},\,\forall \,t,\,s>0$,

2. $\int_X|T_t f(x)|^2\,d\mu(x)\le \int_X f(x)^2\,d\mu(x),\,\forall
\,f\in L^2(X,\mu) \ \mbox{and} \ \forall \, t>0$,

3. $T_t f\to f$ in $L^2(X,\mu)$ when $t\to 0$,

4. if $f\in L^2(X,\mu)$ satisfies $0\le f\le C$, then $0\le T_tf\le
C$ for all $t>0$,

5. if $f\in \mathscr{D}(A)$, then $\frac 1t (T_tf-f)\to Af$ in
$L^2(X,\mu)$ as $t\to 0$, and

6. $A T_tf=\frac {\pa}{\pa_t}T_tf$, $\forall t>0$ and $\forall \,
f\in L^2(X,\mu)$.

A measurable function $p:\, \rr\times X\times X \to [0,\fz]$ is said
to be a heat kernel on $X$ if
$$ T_tf(x)=\int_X f(y)p(t,x,y)\,d\mu(y)$$ for
every $f\in L^2(X,\mu)$ and all $t\ge 0$, and $p(t,x,y)=0$ for every
$t<0$. Let the measure on $X$ be doubling and support a
2-Poincar\'e inequality. Sturm (\cite{st3}) proved the existence of the heat
kernel, and a Gaussian estimate for the heat kernel which in our settings reads
as: there exist positive constants $C,\,C_1,\,C_2$ such that
\begin{equation}\label{2.5}
  C^{-1}t^{-\frac Q2}e^{-\frac{d(x,y)^2}{C_2t}}\le p(t,x,y)\le
  Ct^{-\frac Q2}e^{-\frac{d(x,y)^2}{C_1t}}.
\end{equation}
Moreover, the heat kernel is
proved in \cite{st1} to be a probability measure, i.e., for each $x\in X$ and
$t>0$,
\begin{equation}\label{2.6}
T_t1(x)=\int_X p(t,x,y)\,d\mu(y)=1.
\end{equation}

The following Lemma \ref{l2.4}, Lemma \ref{l2.5} and Proposition
\ref{p2.1} were established in \cite{krs} for $Q>1$. However, their proofs
show that they also hold for all $Q\ge 1$. We omit the details here.

\begin{lem}\label{l2.4}
Let $T>0$. Then for $\mu$-almost every $x\in X$,
$D_yp(\cdot,x,\cdot)\in L^2([0,T]\times X)$ and there exists a
positive constant $C_{T,x}$, depending on $T$ and $x$, such that
$$\int_0^T \int_{X}|D_yp(t,x,y)|^2
\,d\mu(y)\,dt\le C_{T,x}.$$
\end{lem}

\begin{lem}\label{l2.5}
There exists $C>0$ such that for all $0<T<r^3<1$ and every $x\in X$ it holds
$$\int_0^T \int_{B(x,2r)\setminus B(x,2r)}|D_yp(t,x,y)|^2
\,d\mu(y)\,dt\le e^{-CT^{-1/3}}.$$
\end{lem}

The following result shows that the heat kernel plays the role of
a fundamental solution. Let us recall the definition of test functions in this subject. The
test functions are continuous functions $\phi:\,[0,T]\times X\to
\rr$ such that for every fixed $t>0$, $\phi(t,\cdot)=\phi_t(\cdot)\in
N^{1,2}(X)$, $D_y \phi(t,y)\in L^2([0,T]\times X)$, and for
$\mu$-almost every $x\in X$, $\phi(\cdot,x)$ is absolutely
continuous on $[0,T]$. Moreover, we assume that there is a constant
$\dz_0=\dz_0(x)>0$ such that the following H\"older continuity
property holds for $\phi$ and the ``center point" $x$ of the heat
kernel $p(\cdot,x,\cdot)$: there exist $C$ and $\az>0$ such that for
every $\dz<\dz_0$ and for all $(t,y)\in [0,\dz]\times B(x,\dz)$,
\begin{equation}\label{2.7}
|\phi(t,y)-\phi(0,x)|\le C\dz^\az.
\end{equation}

\begin{prop}\label{p2.1}
There exists a constant $K$ such that for every test function $\phi_t(y)$ and every $x\in X$
it holds
\begin{eqnarray*}
\int_0^T\int_X \phi_t(y)A_yp(t,x,y)\,d\mu(y)\,dt&& :=-\int_0^T\int_X
D\phi_t(y)\cdot D_yp(t,x,y)\,d\mu(y)\,dt\\
&&=\int_0^T\int_X
\phi_t(y)\frac{\pa}{\pa t}p(t,x,y)\,d\mu(y)\,dt+K\phi_0(x).\nonumber
\end{eqnarray*}
\end{prop}

\section{Proof of Theorem \ref{t1}}
\begin{proof} Our proof is developed from the proof of \cite[Theorem 1]{krs};
see also \cite{ck}. The proof is quite long, so that we
first describe the strategy. Let $u$ be a solution of \eqref{1.3} on a domain
$\Omega\subseteq X$ and let $r\in (0,1)$ and $B=B(y_0,6r)\subset\subset \Omega$.
Following \cite{krs}, we want to bound $|Du(x_0)|$ for every
$x_0\in B(y_0,r)\setminus A$, where $A$ is a set of measure zero
depending on $u$ and the heat kernel.

Let $\phi$ be a Lipschitz function on $X$ such that $\phi\equiv 1$ on
$B(x_0,r)$, $\supp \phi\subseteq B(x_0,2r)$ and $|D\phi|\le C/r$.
Now fix $T<r^3$ and for every $t\in (0,T)$, set
$$w(t,x):=u(x)\phi(x)-T_t(u\phi)(x_0).$$
Then $Dw(t,x_0)=D(u\phi)(x_0)=Du(x_0)$.

Let $\ez\in (0,1)$ be determined in the future. Now, for $t\in (0,T]$, let
\begin{eqnarray}\label{3.1}
J(t):=&&\frac{1+t^\ez}{t}\int_0^t\int_X |Dw(s,x)|^2 p(s,x_0,x)\,d\mu(x)\,ds\nonumber\\
&&+ \frac{1+t^\ez}{t}\int_0^t\int_X w(s,x)\phi(x) Au(x)p(s,x_0,x)\,d\mu(x)\,ds.
\end{eqnarray}
From the following proofs, we will see that $\ez$ needs to be chosen depending on $p,Q$,
and that the term of $\ez$ plays a crucial role in proving the theorem.

We will prove the following results.

\begin{prop}\label{p3.1}
There exists a positive constant $C_{T,r}$, depending on $T,r$,
such that
\begin{equation*}
J(T)\le C_{T,r}(\|u\|_{L^2(B(x_0,4r))}+\|f\|_{L^p(B(x_0,4r))})^2.
\end{equation*}
\end{prop}

\begin{prop}\label{p3.2}
There exists a positive constant $C_{T,r}$, depending on $T,r$,
such that
\begin{equation*}
\int_0^T \frac{\,d}{\,dt}J(t)\,dt\ge -C_{T,r}(\|u\|_{L^2(B(x_0,4r))}+\|f\|_{L^p(B(x_0,4r))})^2.
\end{equation*}
\end{prop}

\begin{prop}\label{p3.3}
For almost every $x_0\in B(y_0,r)$ it holds $\lim_{t\to 0^+}J(t)=|Du(x_0)|.$
\end{prop}

Combining Propositions \ref{p3.1}, \ref{p3.2} and \ref{p3.3},
we finally obtain
\begin{eqnarray*}
|D(u)(x_0)|^2&&=J(T)-\int_0^T \frac{\,d}{\,dt}J(t)\,dt\\
&&\le C_{T,r}(\|u\|_{L^2(B(x_0,4r))}
+\|f\|_{L^p(B(x_0,4r))})^2,
\end{eqnarray*}
for almost every $x_0\in B(y_0,r)$.

Now for all $x,y\in B(y_0,r)\subseteq \Omega$ with $B(y_0,6r)\subset\subset\Omega$,
let $B_1=B(x,d(x,y))$ and $B_{-1}=B(y,d(x,y))$. For $j\ge 1$,
set $B_j=2^{-1}B_{i-1}$ and $B_{-j}=2^{-1}B_{-j+1}$ inductively.
Further, if $x,y$ are Lebesgue points of $u$, then
$$|u(x)-u(y)|\le \sum_{j=-\fz}^{\fz}|u_{B_j}-u_{B_{j+1}}|,$$
where for each $j$, the Poincar\'e inequality yields
\begin{eqnarray*}
|u_{B_j}-u_{B_{j+1}}|&&\le C \mathrm{diam}(2B_j)
\lf(\fint_{2B_j} |Du(x)|^2\,d\mu(x)\r)^{1/2}\\
&&\le C_{T,r}\mathrm{diam}(2B_j)(\|u\|_{L^2(B(y_0,6r))}
+\|f\|_{L^p(B(y_0,6r))}).
\end{eqnarray*}
Hence, we obtain
\begin{eqnarray*}
|u(x)-u(y)|\le C_{T,r}(\|u\|_{L^2(B(y_0,6r))}
+\|f\|_{L^p(B(y_0,6r))})d(x,y),
\end{eqnarray*}
for almost all $x,y\in B(y_0,r)$.  Then
$u$ can be extended to a locally Lipschitz continuous function on $\Omega$, which completes
the proof of Theorem \ref{t1}.
\end{proof}

Let us prove the Propositions.

\begin{proof}[Proof of Proposition \ref{p3.1}]
Since $w(t,x)=u(x)\phi(x)-T_t(u\phi)(x_0),$ we have
$$|D(u\phi)|^2=|Dw|^2=\frac 12 A w^2-w(u A\phi+\phi Au+2Du\cdot D\phi)$$
in the weak sense of measures. Also, in what follows we extend $A$ formally
to all of $N^{1,2}(X)$ by defining
\begin{equation}\label{3.2}
\int_X v(x) Au(x)\,d\mu(x)=-\int_X Dv(x)\cdot Du(x)\,d\mu(x)=\int_X Av(x) u(x)\,d\mu(x).
\end{equation}

Moreover, we set $m(t)=T_t(u\phi)(x_0)$. Then
$\frac{\pa}{\pa t}w^2=2w\frac{\pa}{\pa t}w=-2wm'(t)$, which further implies that
$$|Dw|^2=\frac 12 \lf(A+\frac{\pa}{\pa t}\r)w^2-w(u A\phi+\phi Au
+2Du\cdot D\phi-m'(t))$$
in the weak sense of measures.
Thus, we obtain
\begin{eqnarray}\label{3.3}
&&\int_0^t\int_X [|Dw(s,x)|^2+w(s,x)\phi(x)Au(x)] p(s,x_0,x)\,d\mu(x)\,ds\nonumber\\
&&\hs =\frac 12 \int_0^t\int_X \lf(A+\frac{\pa}{\pa s}\r)
w^2(s,x)p(s,x_0,x)\,d\mu(x)\,ds\nonumber\\
&&\hs\hs-\int_0^t\int_X w(s,x)[u(x) A\phi(x)+
2Du(x)\cdot D\phi(x) -m'(s)]p(s,x_0,x)\,d\mu(x)\,ds.
\end{eqnarray}
Recall that for each $s>0$ and $x_0\in X$, $T_s(1)(x_0)=\int_X p(s,x_0,x)\,d\mu(x)=1$.
We then have
\begin{eqnarray*}
&&\int_0^t\int_X w(s,x)m'(s)p(s,x_0,x)\,d\mu(x)\,ds
=\int_0^tm'(s)T_s(u\phi)(x_0)(1-T_s(1)(x_0))\,ds=0.
\end{eqnarray*}

Now applying integration by parts and using \eqref{3.2}, we obtain
\begin{eqnarray*}
&&\int_0^t\int_X \lf(A+\frac{\pa}{\pa s}\r)
w^2(s,x)p(s,x_0,x)\,d\mu(x)\,ds\\
&&\hs =\int_0^t\int_X w^2(s,x)\lf(A-\frac{\pa}{\pa s}\r)p(s,x_0,x)\,d\mu(x)\\
&&\hs\hs +\int_X w^2(t,x)p(t,x_0,x)\,d\mu(x)-\int_X w^2(0,x)p(0,x_0,x)\,d\mu(x).
\end{eqnarray*}
At this point, we want to use Proposition \ref{p2.1} with $\phi(t,x)=w^2(t,x)$.
By Lemma \ref{l2.1} and $u\in N^{1,2}_{\loc}(X)$,
we have $D_x(w^2(t,x))=2w(t,x)D(u\phi)(x)\in L^2([0,T]\times X)$.
The property (6) of our semigroup guarantees the continuity of $w^2(t,x)$ on $[0,T]$.
Since $w^2(t,x)$ may equal to a constant outside $B(x_0,2r)$, it may
not be in $L^2(X)$, but we always have $D(w^2(t,\cdot))\in L^2(X)$ which is enough
for us to use Proposition \ref{p2.1}. The only thing left is to verify the H\"older
continuity of $w^2(t,x)$.

For $p>Q\ge 2$, by \cite[Theorem 5.13]{bm},
we see that $u$ is locally H\"older continuous; while
for $1\le Q<2$, since $u\phi\in N^{1,2}(X)$, by \cite[Theorem 5.1]{sh},
$u\phi$ is H\"older continuous with exponent $1-\frac Q2$.
More precisely, for almost all $x,y\in B$, we have
$$|u(x)-u(y)|\le C_r[\|u\|_{L^2(B(x_0,4r))}+\|g\|_{L^p(B(x_0,4r))}]d(x,y)^\dz.$$
Notice here that, when $Q< 2$, $\dz=1-\frac Q2$.
In what follows, for simplicity, we define
$$C(u,f):=\|u\|_{L^2(B(x_0,4r))}+\|f\|_{L^p(B(x_0,4r))}.$$

In the following proof, we will repeatedly use the fact
that for fixed $\bz,\gz\in (0,\fz)$, $t^{\bz}e^{-t^\gz}$ and $t^{-\bz}e^{-t^{-\gz}}$
are bounded on $(0,\fz)$.
Let us now show that \eqref{2.7} holds for $w$. By the local H\"older continuity of $u\phi$
and \eqref{2.5}, we have
\begin{eqnarray}\label{3.4}
&&|w(t,x)|=|u(x)\phi(x)-T_t(u\phi)(x_0)|\nonumber\\
&&\hs=|u(x)\phi(x)-u(x_0)\phi(x_0)+u(x_0)\phi(x_0)-T_t(u\phi)(x_0)|\nonumber\\
&&\hs \le CC(u,f)d(x,x_0)^\dz+\int_X |u(x_0)\phi(x_0)-u(x)\phi(x)|p(t,x_0,x)\,d\mu(x)\nonumber\\
&&\hs \le CC(u,f)\lf\{d(x,x_0)^\dz+\int_X d(x,x_0)^\dz t^{-\frac Q2}
e^{-\frac{d(x,x_0)^2}{2C_1t}} e^{-\frac{d(x,x_0)^2}{2C_1t}}\,d\mu(x)\r\}\nonumber\\
&&\hs \le CC(u,f)\lf\{d(x,x_0)^\dz+t^{\dz/2}\int_X  p(lt,x_0,x)\,d\mu(x)\r\}\nonumber\\
&&\hs \le CC(u,f)(d(x,x_0)^\dz+t^{\dz/2}),
\end{eqnarray}
where $l=\frac{2C_1}{C_2}$. Now let $\gz\in (0,T]$ and $(t,x)\in [0,\gz]\times B(x_0,\gz)$. Then
by \eqref{3.4} and the fact that $w(0,x_0)=0$, we see that
\begin{eqnarray*}
|w(t,x)-w(0,x_0)|\le CC(u,f)(d(x,x_0)^\dz+t^{\dz/2})\le CC(u,f)\gz^{\dz/2}.
\end{eqnarray*}

Thus, this allows us to use Proposition \ref{p2.1} to obtain
\begin{eqnarray*}
&&\int_0^t\int_X \lf(A+\frac{\pa}{\pa s}\r)
w^2(s,x)p(s,x_0,x)\,d\mu(x)\,ds\\
&&\hs=Kw^2(0,x_0)+\int_X w^2(t,x)p(t,x_0,x)\,d\mu(x)-\int_X w^2(0,x)p(0,x_0,x)\,d\mu(x)\nonumber.
\end{eqnarray*}
Using \eqref{2.6} gives
\begin{eqnarray*}
\int_X w^2(0,x)p(0,x_0,x)\,d\mu(x)&&=\lim_{s\to 0^+}\int_X w^2(s,x)p(s,x_0,x)\,d\mu(x)\\
&&=\lim_{s\to 0^+} [T_s((u\phi)^2)(x_0)-(T_s(u\phi)(x_0))^2]=0,
\end{eqnarray*}
which together with the fact $w(0,x_0)=0$ yields
\begin{eqnarray}\label{3.5}
&&\int_0^t\int_X \lf(A+\frac{\pa}{\pa s}\r)
w^2(s,x)p(s,x_0,x)\,d\mu(x)\,ds=\int_X w^2(t,x)p(t,x_0,x)\,d\mu(x).
\end{eqnarray}

We now estimate the second term in \eqref{3.3}. Recall that
$\phi=1$ on $B(x_0,r)$ and $\supp \phi \subseteq B(x_0,2r)$.
By Lemma \ref{l2.1}, Lemma \ref{l2.2}, Lemma \ref{l2.5} and the H\"older
inequality, we obtain
\begin{eqnarray}\label{3.x1}
&&\lf|\int_0^t\int_X w(s,x)u(x) A\phi(x)p(s,x_0,x)\,d\mu(x)\,ds\r|\nonumber\\
&&\hs=\lf|\int_0^t\int_X D\lf(w(s,x)u(x)p(s,x_0,x)\r)\cdot D\phi(x)
\,d\mu(x)\,ds\r|\nonumber\\
&&\hs\le\lf|\int_0^t\int_X w(s,x)u(x)Dp(s,x_0,x)\cdot D\phi(x)
\,d\mu(x)\,ds\r|\nonumber\\
&&\hs\hs+\lf|\int_0^t\int_X w(s,x)p(s,x_0,x)Du(x)\cdot D\phi(x)
\,d\mu(x)\,ds\r|\nonumber\\
&&\hs\hs+\lf|\int_0^t\int_X u(x)p(s,x_0,x)Dw(s,x)\cdot D\phi(x)
\,d\mu(x)\,ds\r|\nonumber\\
&&\hs\le C_r\|u\|_{L^\fz(B(x_0,2r))}^2
\lf(\int_0^t\int_{B(x_0,2r)\setminus B(x_0,r)} |Dp(s,x_0,x)|^2
\,d\mu(x)\,ds\r)^{1/2}\nonumber\\
&&\hs\hs+C_r e^{-Ct^{-\frac 13}}\|u\|_{L^\fz(B(x_0,2r))}
\lf(\int_{B(x_0,2r)\setminus B(x_0,r)} |Du(x)|^2\,d\mu(x)\r)^{1/2}\nonumber\\
%&&\hs\le C_r e^{-Ct^{-\frac 13}}\|u\|_{L^\fz(B(x_0,2r))}\bigg[
%\|u\|_{L^\fz(B(x_0,2r))}\nonumber\\
%&&\hs\hs +\|f\|^{1/2}_{L^p(B(x_0,4r))}\|u\|^{1/2}_{L^2(B(x_0,4r))}
%+\|u\|_{L^2(B(x_0,4r))}\bigg]\nonumber\\
&&\hs\le C_r e^{-Ct^{-\frac 13}}C(u,f)^2.
\end{eqnarray}
Similarly we have
\begin{eqnarray}\label{3.6}
&&\lf|\int_0^t\int_X w(s,x)p(s,x_0,x)Du(x)\cdot D\phi(x)\,d\mu(x)\,ds\r|\le C_r C(u,f)^2 e^{-Ct^{-\frac 13}}.
\end{eqnarray}

Combining the estimates \eqref{3.5}-\eqref{3.6}, by \eqref{3.3}, we obtain that $t\in (0,T]$,
\begin{eqnarray}\label{3.x2}
J(t)&&\le \frac{1+t^\ez}{2t}\lf|\int_0^t\int_X \lf(A+\frac{\pa}{\pa s}\r)
w^2(s,x)p(s,x_0,x)\,d\mu(x)\,ds\r|\nonumber\\
&&\hs+\frac{1+t^\ez}{t}\lf|\int_0^t\int_X w(s,x)[u(x) A\phi(x)
+2Du(x)\cdot D\phi(x)p(s,x_0,x)\,d\mu(x)\,ds\r|\nonumber\\
&&\le \frac{1+t^\ez}{2t}\int_X w^2(t,x)p(t,x_0,x)\,d\mu(x)
+C_re^{-Ct^{-\frac 13}}C(u,f)^2.
\end{eqnarray}
By letting $t=T$ and using Lemma \ref{l2.1}, we obtain
\begin{eqnarray*}
  J(T)\le C_T \|u\phi\|^2_{L^\fz(X)}+C_{T,r}C(u,f)^2\le C_{T,r}C(u,f)^2,
\end{eqnarray*}
which is the desired estimate, and hence completes the proof of Proposition \ref{p3.1}.
\end{proof}

\begin{proof}[Proof of Proposition \ref{p3.2}]
Let us now estimate the derivative $J'(t)=\frac{\,d}{\,dt}J(t)$.
By Lemma \ref{l2.4}, for almost every $x_0\in X$,
$D_yp(s,x_0,\cdot)\in L^2(X)$. From this together with the fact that
for each fixed $s$, $w(s,\cdot),\phi, p(s,x_0,\cdot)$ are bounded functions,
we see that $w\phi p\in N_0^{1,2}(B(x_0,2r))$. Thus by \eqref{3.2}, we obtain
\begin{eqnarray}\label{3.x3}
&&\int_0^t\int_X w(s,x)\phi(x)Au(x)p(s,x_0,x)\,d\mu(x)\nonumber\\
&&\hs=-\int_0^t\int_X D(w(s,\cdot)\phi p(s,x_0,\cdot))(x)\cdot Du(x)\,d\mu(x)\nonumber\\
&&\hs=\int_0^t\int_X w(s,x)\phi(x)f(x)p(s,x_0,x)\,d\mu(x).
\end{eqnarray}
This and \eqref{3.x2} further imply that
\begin{eqnarray}\label{3.8}
\frac{\,d}{\,dt}J(t)&&=\lf(-\frac 1 {t^2}-\frac {1-\ez}{t^{2-\ez}}\r)
\frac{t}{1+t^\ez}J(t)+\frac{1+t^\ez}{t}
\int_X |Dw(t,x)|^2 p(t,x_0,x)\,d\mu(x)\nonumber\\
&&\hs+\frac{1+t^\ez}{t} \int_X w(t,x)\phi(x) f(x)p(t,x_0,x)\,d\mu(x)\nonumber\\
&&\ge \frac{1+(1-\ez)t^\ez}{t}\lf(\int_X |Dw(t,x)|^2 p(t,x_0,x)\,d\mu(x)-
\frac 1{2t}\int_X w^2(t,x)p(t,x_0,x)\,d\mu(x)\r)\nonumber\\
&&\hs-C_r e^{-Ct^{-\frac 13}}C(u,f)^2+\frac{\ez t^\ez}{t}
\int_X |Dw(t,x)|^2 p(t,x_0,x)\,d\mu(x)\nonumber\\
&&\hs+\frac{1+t^\ez}{t} \int_X w(t,x)\phi(x) f(x)p(t,x_0,x)\,d\mu(x).
\end{eqnarray}

The main difficulty left is to estimate the last term in \eqref{3.8}.
To this end, we now decompose our proof into two different
cases: (i) $Q\ge 2$ and (ii) $Q\in [1,2)$. From the following proof,
we will see that the term $t^\ez$ in \eqref{3.1} plays a key role in
matching the two terms $\int_X |Dw(t,x)|^2 p(t,x_0,x)\,d\mu(x)$ and
$\frac 1{2t}\int_X w^2(t,x)p(t,x_0,x)\,d\mu(x)$, which allows us to use \eqref{1.1}.

{\bf Case (i) $Q\ge 2$.} Recall that $\ez\in (0,1)$.
Applying the Young inequality to the last term in \eqref{3.8}
and choosing suitable constants, we obtain
\begin{eqnarray*}
\frac{\,d}{\,dt}J(t)&&\ge \frac{1+(1-\ez)t^\ez}{t}\lf(\int_X |Dw(t,x)|^2 p(t,x_0,x)\,d\mu(x)-
\frac 1{2t}\int_X w^2(t,x)p(t,x_0,x)\,d\mu(x)\r)\nonumber\\
&&\hs-C_r e^{-Ct^{-\frac 13}}C(u,f)+\frac{\ez t^\ez}{t}
\int_X |Dw(t,x)|^2 p(t,x_0,x)\,d\mu(x)\nonumber\\
&&\hs-\frac{\ez t^\ez}{2t^2} \int_X w^2(t,x)p(t,x_0,x)\,d\mu(x)
-\frac{C_{T,\ez}}{t^\ez}\int_X (\phi(x)f(x))^2p(t,x_0,x)\,d\mu(x)\nonumber\\
&&\ge \frac{1+t^\ez}{t}\lf(\int_X |Dw(t,x)|^2 p(t,x_0,x)\,d\mu(x)-
\frac 1{2t}\int_X w^2(t,x)p(t,x_0,x)\,d\mu(x)\r)\nonumber\\
&&\hs-C_r e^{-Ct^{-\frac 13}}C(u,f)
-\frac{C_{T,\ez}}{t^\ez}\int_X (\phi(x)f(x))^2p(t,x_0,x)\,d\mu(x).
\end{eqnarray*}

For each fixed $t\in (0,T)$, either
$$\int_X |Dw(t,x)|^2 p(t,x_0,x)\,d\mu(x)\ge
\frac 1{2t}\int_X w^2(t,x)p(t,x_0,x)\,d\mu(x)$$
or
$$\int_X |Dw(t,x)|^2 p(t,x_0,x)\,d\mu(x)<
\frac 1{2t}\int_X w^2(t,x)p(t,x_0,x)\,d\mu(x).$$
In the first case, we have
\begin{eqnarray}\label{3.9}
\frac{\,d}{\,dt}J(t)\ge -C_r e^{-Ct^{-\frac 13}}C(u,f)
-\frac{C_{T,\ez}}{t^\ez}\int_X (\phi(x)f(x))^2p(t,x_0,x)\,d\mu(x).
\end{eqnarray}
Let us consider the second case. By \eqref{2.7}, \eqref{3.4} and the fact that
$\frac{d(x,x_0)^\dz}{t^{\dz/2}}e^{-\frac{|x-x_0|^2}{2C_1t}}$ is bounded, we obtain
\begin{eqnarray*}
&&\int_X |Dw(t,x)|^2 p(t,x_0,x)\,d\mu(x)< \frac 1{2t}\int_X w^2(t,x)p(t,x_0,x)\,d\mu(x)\nonumber\\
&&\hs\le CC(u,f)^2\frac 1{2t}\int_{X} (d(x,x_0)^\dz+t^{\dz/2})^2t^{-Q/2}
e^{-\frac{|x-x_0|^2}{C_1t}}\,d\mu(x)\\
&&\hs \le CC(u,f)^2t^{\dz-1}\int_{X} p(lt,x_0,x)\,d\mu(x)\le CC(u,f)^2t^{\dz-1},\nonumber
\end{eqnarray*}
where $l=\frac{C_1}{2C_2}$. The fact $w=u\phi-T_t(u\phi)(x_0)$ implies
\begin{eqnarray*}
\int_X w^2(t,x)p(t,x_0,x)\,d\mu(x)&&=\int_X (u(x)\phi(x))^2p(t,x_0,x)\,d\mu(x)\\
&&\hs -\lf(\int_X u(x)\phi(x)p(t,x_0,x)\,d\mu(x)\r)^2.
\end{eqnarray*}
Then, by \eqref{1.1} with $g$ replaced by $u\phi$, we obtain
\begin{eqnarray}\label{3.10}
\frac{\,d}{\,dt}J(t)&&\ge -C(1+t^\ez)\int_X |Dw(t,x)|^2 p(t,x_0,x)\,d\mu(x)\nonumber\\
&&\hs-C_r e^{-Ct^{-\frac 13}}C(u,f)
-\frac{C_T}{t^\ez}\int_X (\phi(x)f(x))^2p(t,x_0,x)\,d\mu(x)\nonumber\\
&&\ge -C_{T,r}C(u,f)^2t^{\dz-1}
-\frac{C_{T,\ez}}{t^\ez}\int_X (\phi(x)f(x))^2p(t,x_0,x)\,d\mu(x).
\end{eqnarray}
Thus, from \eqref{3.9} and \eqref{3.10},
we see that \eqref{3.10} holds in both cases.

Since $p>Q\ge 2$, we may choose $\ez\in (0,1)$
such that $Q/p+\ez<1$. This together with the H\"older inequality
implies that
\begin{eqnarray*}
&&\int_0^T \frac{\,d}{\,dt}J(t)\,dt\nonumber\\
&&\hs\ge
-C_{T,r}C(u,f)^2\int_0^T t^{\dz-1}\,dt-\int_0^T\frac{C_{T}}{t^\ez}\int_X (\phi(x) f(x))^2
p(t,x_0,x)\,d\mu(x)\,dt\nonumber\\
&&\hs\ge -C_{T,r}C(u,f)^2\lf[1+\int_0^T
t^{-\ez}\lf(\int_X p(t,x_0,x)^{\frac{p}{p-2}}\,d\mu(x)\r)^{1-\frac 2p}\,dt\r]\nonumber\\
&&\hs\ge -C_{T,r}C(u,f)^2\int_0^T
t^{-\ez-\frac Qp}\lf(\int_X p(t,x_0,x)\,d\mu(x)\r)^{1-\frac 2p}\,dt\nonumber\\
&&\hs\ge -C_{T,r}C(u,f)^2,
\end{eqnarray*}
which completes the proof of the case $Q\in [2,\fz).$

{\bf Case (ii) $Q\in [1,2).$}
Let us first estimate the last term in \eqref{3.8}. Let $\az\in (0,\frac 12)$
be fixed in what follows. Choose $\psi_t(x)$ to be a Lipschitz function on $X$
such that $\psi_t(x)\equiv 1$ on
$B(x_0,t^\az)$, $\supp \psi_t\subseteq B(x_0,2t^\az)$ and $|D\phi|\le Ct^{-\az}$. Write
\begin{eqnarray*}
&&\frac{1+t^\ez}{t} \lf|\int_X w(t,x)\phi(x) f(x)p(t,x_0,x)\,d\mu(x)\r|\\
&&\hs\le \frac{C_T}{t} \int_{B(x_0,t^\az)} \lf|\psi_t(x) w(t,x)\phi(x) f(x)p(t,x_0,x)\r|\,d\mu(x)\nonumber\\
&&\hs\hs+\frac{C_T}{t} \int_{(B(x_0,t^\az))^\com} \lf|w(t,x)\phi(x) f(x)p(t,x_0,x)\r|\,d\mu(x)\nonumber\\
&&\hs =: \mathrm{H}_1+\mathrm{H}_2.\nonumber
\end{eqnarray*}

Since $\az\in (0,\frac 12)$, $t^{-1-\frac Q2}e^{-\frac{t^{2\az-1}}{2C_1}}$
is bounded on $(0,\fz)$. This, together with the H\"older inequality and \eqref{2.5} yields
\begin{eqnarray*}
\mathrm{H}_2&&\le \frac{C_T}{t}\|u\|_{L^\fz(B(x_0,2r))}
\int_{(B(x_0,t^\az))^\com} |\phi(x)f(x)|t^{-\frac Q2}e^{-\frac{d(x,x_0)^2}{C_1t}}\,d\mu(x)\\
&&\le C_{T,r}e^{-ct^{2\az-1}}C(u,f)^2.\nonumber
\end{eqnarray*}

Let us estimate the term $\mathrm{H}_1$. Let $1<s<\min\{2,p\}$ and let $s'$
be the conjugate of $s$, i.e., $\frac 1s+\frac 1{s'}=1$. By the H\"older inequality,
we have
\begin{eqnarray}\label{3.11}
\mathrm{H}_1&&\le \frac{C_T}{t} \|f\phi\|_{L^{s}(B(x_0,t^\az))}
\|\psi_t w(t,\cdot)p(t,x_0,\cdot)\|_{L^{s'}(B(x_0,t^\az))}\nonumber\\
&&\le \frac{C_T}{t} \|f\|_{L^{p}(B(x_0,t^\az))}\mu(B(x_0,t^\az)))^{\frac{1}{s}-\frac 1p}
\|\psi_t w(t,\cdot)p(t,x_0,\cdot)\|_{L^{s'}(B(x_0,t^\az))}.
\end{eqnarray}
Notice that for each $t$, $\psi_t w(t,\cdot)p(t,x_0,\cdot)\in N^{1,2}_0(B(x_0,2t^\az))$. Then by
the Sobolev-Poincar\'e inequality \eqref{2.1}, we obtain
\begin{eqnarray}\label{3.12}
&&\|\psi_t w(t,\cdot) p(t,x_0,\cdot)\|_{L^{s'}(B(x_0,2t^\az))}\nonumber\\
&&\hs\le 2t^{\az}\mu(B(x_0,2t^\az)))^{\frac 1{s'}-\frac 12}\|D(\psi_t w(t,\cdot)
p(t,x_0,\cdot))\|_{L^2(B(x_0,2t^\az))}.
\end{eqnarray}
Let us estimate $\|D(\psi_t w(t,\cdot)p(t,x_0,\cdot))\|_{L^2(B(x_0,2t^\az))}$. Applying the Leibniz rule,
the Gaussian bounds of heat kernel \eqref{2.5}, \eqref{2.6} and \eqref{3.4},  we obtain
\begin{eqnarray}\label{3.13}
&&\|D(\psi_t w(t,\cdot)p(t,x_0,\cdot))\|_{L^2(B(x_0,2t^\az))}\nonumber\\
&&\hs\le C\|t^{-\az}w(t,\cdot) p(t,x_0,\cdot)\|_{L^2({B(x_0,2t^\az)})}
+\||Dw(t,\cdot)| p(t,x_0,\cdot)\|_{L^2({B(x_0,2t^\az)})}\nonumber\\
&&\hs\hs+C_r\|u\|_{L^\fz(B(x_0,2r))}\|t^{\az(1-\frac Q2)} D_yp(t,x_0,\cdot)\|_{L^2({B(x_0,2t^\az)})}\nonumber\\
&&\hs\le Ct^{-\az-\frac Q4}\lf(\int_{B(x_0,2t^\az)}w^2(t,x) p(t,x_0,x)\,d\mu(x)\r)^{1/2}\nonumber\\
&&\hs\hs+Ct^{-\frac Q4}\lf(\int_{B(x_0,2t^\az)}|Dw(t,x)|^2 p(t,x_0,x)\,d\mu(x)\r)^{1/2}\nonumber\\
&&\hs\hs+C_rC(u,f)t^{\az(1-\frac Q2)}\|D_yp(t,x_0,\cdot)\|_{L^2(B(x_0,2t^\az))}.
\end{eqnarray}

Combining the estimates \eqref{3.11}-\eqref{3.13}, by using the Young inequality, we obtain
\begin{eqnarray*}
\mathrm{H}_1&&\le C_{T,r} t^{\az-1+\az Q(\frac 12-\frac 1p)}\|f\|_{L^{p}(B(x_0,t^\az))}
\bigg\{t^{-\az-\frac Q4}\lf(\int_{B(x_0,2t^\az)}w^2(t,x) p(t,x_0,x)\,d\mu(x)\r)^{1/2}\\
&&\hs+t^{-\frac Q4}\lf(\int_{B(x_0,2t^\az)}|Dw(t,x)|^2 p(t,x_0,x)\,d\mu(x)\r)^{1/2}\\
&&\hs+C(u,f)t^{\az(1-\frac Q2)}\|D_yp(t,x_0,\cdot)\|_{L^2(B(x_0,2t^\az))}\bigg\}\\
&&\le C_Tt^{2\az Q(\frac 12-\frac 1p)-\frac Q2-\ez}\|f\|^2_{L^p(B(x_0,2r))}
+\frac{\ez t^{\ez}}{4t^2}\int_{B(x_0,2t^\az)}|w^2(t,x)| p(t,x_0,x)\,d\mu(x)\\
&&\hs+ C_T t^{2\az-1+2\az Q(\frac 12-\frac 1p)-\frac Q2-\ez}\|f\|^2_{L^p(B(x_0,2r))}
+\frac{\ez t^{\ez}}{2t}\int_{B(x_0,2t^\az)}|Dw(t,x)|^2 p(t,x_0,x)\,d\mu(x)\\
&&\hs+C_{T,r} t^{2\az-2+2\az Q(\frac 12-\frac 1p)+2\az(1-\frac Q2)}\|f\|^2_{L^p(B(x_0,2r))}
+C(u,f)^2\|D_yp(t,x_0,\cdot)\|^2_{L^2(B(x_0,2t^\az))}\\
&&=:C_Tt^{g_1(\az,Q,p)-\ez}\|f\|^2_{L^p(B(x_0,2r))}
+\frac{\ez t^{\ez}}{4t^2}\int_{B(x_0,2t^\az)}|w^2(t,x)| p(t,x_0,x)\,d\mu(x)\\
&&\hs+ C_T t^{g_2(\az,Q,p)-\ez}\|f\|^2_{L^p(B(x_0,2r))}
+\frac{\ez t^{\ez}}{2t}\int_{B(x_0,2t^\az)}|Dw(t,x)|^2 p(t,x_0,x)\,d\mu(x)\\
&&\hs+C_{T,r} t^{g_3(\az,Q,p)}\|f\|^2_{L^p(B(x_0,2r))}
+C(u,f)^2\|D_yp(t,x_0,\cdot)\|^2_{L^2(B(x_0,2t^\az))}.
\end{eqnarray*}
Since $p>Q$, we have
$$\min_{1\le i\le 3}\lf\{g_i\lf(1/2,Q,p\r)\r\}>-1.$$
Since each $g_i$ is a continuous function of $\az$, there exists $\az\in (\frac 13,\frac 12)$ such that
$$\min_{1\le i\le 3}\lf\{g_i\lf(\az,Q,p\r)\r\}>-1.$$
Fix such an $\az$ and choose
$$\ez\in \lf(0,\frac 12 +\frac 12\min_{1\le i\le 3}\lf\{g_i\lf(\az,Q,p\r)\r\}\r).$$
Then the above estimate reduces to
\begin{eqnarray*}
\mathrm{H}_1&&\le C_{T,r}t^{\ez-1}C(u,f)^2
+\frac{\ez t^{\ez}}{4t^2}\int_{B(x_0,2t^\az)}|w^2(t,x)| p(t,x_0,x)\,d\mu(x)\\
&&\hs+\frac{\ez t^{\ez}}{2t}\int_{B(x_0,2t^\az)}|Dw(t,x)|^2 p(t,x_0,x)\,d\mu(x)
+C(u,f)^2\|D_yp(t,x_0,\cdot)\|^2_{L^2(B(x_0,2t^\az))}.
\end{eqnarray*}

Notice that for fixed $\bz,\gz\in (0,\fz)$, $t^{-\bz}e^{-t^{-\gz}}$
is bounded on $(0,\fz)$.
Applying this and the estimates of $\mathrm{H}_1$ and $\mathrm{H}_2$ to \eqref{3.8} yields
\begin{eqnarray*}
&&\frac{\,d}{\,dt}J(t)\nonumber\\
%&&\hs\ge \frac{1+(1-\ez)t^\ez}{t}\lf(\int_X |Dw(t,x)|^2 p(t,x_0,x)\,d\mu(x)-
%\frac 1{2t}\int_X w^2(t,x)p(t,x_0,x)\,d\mu(x)\r)\nonumber\\
%&&\hs\hs-C_r e^{-Ct^{-\frac 13}}C(u,f)^2+\frac{\ez t^\ez}{t}
%\int_X |Dw(t,x)|^2 p(t,x_0,x)\,d\mu(x)\nonumber\\
%&&\hs\hs-C_{T,r}e^{-ct^{2\az-1}}C(u,f)^2-C_{T,r}t^{\ez-1}C(u,f)^2
%-\frac{\ez t^{\ez}}{4t^2}\int_{B(x_0,2t^\az)}|w^2(t,x)| p(t,x_0,x)\,d\mu(x)\nonumber\\
%&&\hs\hs-\frac{\ez t^{\ez}}{2t}\int_{B(x_0,2t^\az)}|Dw(t,x)|^2 p(t,x_0,x)\,d\mu(x)
%-\|D_yp(t,x_0,\cdot)\|^2_{L^2(B(x_0,2t^\az))}\nonumber\\
&&\hs\ge \frac{1+(1-\frac \ez 2)t^\ez}{t}\lf(\int_X |Dw(t,x)|^2 p(t,x_0,x)\,d\mu(x)-
\frac 1{2t}\int_X w^2(t,x)p(t,x_0,x)\,d\mu(x)\r)\nonumber\\
&&\hs\hs-C_{T,r}C(u,f)^2\lf[t^{\ez-1}-\|D_yp(t,x_0,\cdot)\|^2_{L^2(B(x_0,2t^\az))}\r].
\end{eqnarray*}
The estimates \eqref{3.9}-\eqref{3.10} simplify the above estimate as
\begin{eqnarray*}
&&\frac{\,d}{\,dt}J(t)\ge -C_{T,r}C(u,f)^2
[t^{\frac12(1-\frac Q2)-1}+t^{\ez-1}]-C(u,f)^2\|D_yp(t,x_0,\cdot)\|^2_{L^2(B(x_0,2t^\az))}.
\end{eqnarray*}
Integrating over $t$ on $[0,T]$ we conclude that
\begin{eqnarray}
\int_0^T J'(t)\,dt&&\ge \int_0^T -C_{T,r}C(u,f)^2
[t^{\frac12(1-\frac Q2)-1}+t^{\ez-1}]\,dt\nonumber\\
&&\hs-C(u,f)^2\int_0^T\int_{B(x,2t^\az)}|D_yp(t,x_0,y)|^2\,d\mu(x)\,dt\nonumber\\
&&\ge -C_{T,r}C(u,f)^2-C(u,f)^2\int_0^T\int_{B(x,2t^\az)}|D_yp(t,x_0,y)|^2\,d\mu(x)\,dt.\nonumber
\end{eqnarray}

Let us estimate the last term. To this end, let us recall the following
inequality established in \cite[(13)]{krs}. For every $x\in X$,
\begin{eqnarray*}
  &&\int_{T_0}^{T_1}\int_{B(x,R_1)}|D_yp(t,x,y)|^2\,d\mu(y)\,dt\\
  &&\hs \le C\lf[\frac{1}{(R_2-R_1)^2}+\frac{1}{(T_2-T_1)^2}\r]
  \int_{T_0}^{T_2}\int_{B(x,R_2)}p(t,x,y)^2\,d\mu(y)\,dt,
\end{eqnarray*}
whenever $0<R_1<R_2$ and $0\le T_0<T_1<T_2$, where $C$ is a positive constant
independent of $R_1,R_2,T_0,T_1,T_2$ and $x$.
Since $\az\in(\frac13,\frac 12)$, we have that
$t^\az\le T^{1/3}\le r$ and $B(x_0,2t^\az)\subset B(x_0,2r)$.
By these facts and $Q\in [1,2)$, we obtain
\begin{eqnarray*}
&&\int_0^T\int_{B(x,2t^\az)}|D_yp(t,x_0,y)|^2\,d\mu(x)\,dt\le
\int_0^T\int_{B(x,2r)}|D_yp(t,x_0,y)|^2\,d\mu(x)\,dt\\
&&\hs\le C\lf[\frac{1}{r^2}+\frac{1}{T^2}\r]
\int_{0}^{2T}\int_{B(x,3r)}p(t,x,y)^2\,d\mu(y)\,dt\\
&&\hs\le C_{T,r}\int_{0}^{2T}\int_{B(x,3r)} t^{-\frac Q2}p(t,x,y)d\mu(y)\,dt\le C_{T,r}.
\end{eqnarray*}

Thus in both cases, we obtain
\begin{equation*}
\int_0^T \frac{\,d}{\,dt}J(t)\,dt\ge -C_{T,r}(\|u\|_{L^2(B(x_0,4r))}+\|f\|_{L^p(B(x_0,4r))})^2,
\end{equation*}
which completes the proof of Proposition \ref{p3.2}.
\end{proof}

\begin{proof}[Proof of Proposition \ref{p3.3}]
By \eqref{3.x3}, \eqref{2.5} and \eqref{3.4},
we see that
\begin{eqnarray*}
&&\lf|\int_0^t\int_X w(s,x)p(s,x_0,x)\phi(x)Au(x)\,d\mu(x)\,ds\r|\nonumber\\
&&\hs\le CC(u,f)\int_0^t\int_{X}
(d(x,x_0)^\dz+s^{\dz/2})s^{-\frac Q2}e^{-\frac{d(x,x_0)^2}{C_1s}}
\lf|\phi(x)f(x)\r|\,d\mu(x)\,ds\nonumber\\
&&\hs\le CC(u,f)\int_0^t s^{\dz/2}T_{ls}(|\phi f|)(x_0)\,ds,
\end{eqnarray*}
where $l=\frac{C_1}{2C_2}$. By the fact that $T_t-I\to 0$ in the strong
operator topology as $t\to 0$, we obtain
\begin{eqnarray}\label{3.14}
&&\lim_{t\to 0^+}\lf|\frac{1+t^\ez}{t}
\int_0^t\int_X w(s,x)p(s,x_0,x)\phi(x)Au(x)\,d\mu(x)\,ds\r|\nonumber\\
&&\hs\le C_T\lim_{t\to 0^+} \frac 1t\int_0^t s^{\dz/2}T_{ls}(|\phi f|)(x_0)\,ds=
C_T\lim_{s\to 0^+}s^{\dz/2}T_{ls}(|\phi f|)(x_0)=0,
\end{eqnarray}
for almost every $x_0\in \Omega$.

By \eqref{3.14},  we further obtain
\begin{eqnarray*}
\lim_{t\to 0^+}J(t)&&=\lim_{t\to 0^+}\frac{1+t^\ez}{t}
\int_0^t\int_X |Dw(s,x)|^2 p(s,x_0,x)\,d\mu(x)\,ds\\
&&=\lim_{s\to 0^+} T_s(|D(u\phi)|^2)(x_0)=|D(u)(x_0)|^2
\end{eqnarray*}
for almost every $x_0\in \Omega$, proving the Proposition.
\end{proof}

\section{Some applications}
\hskip\parindent In this section, we discuss some sufficient conditions
for \eqref{1.1}. As pointed out in the introduction,
the logarithmic inequality \eqref{1.2} guarantees \eqref{1.1};
see \cite{be2,gro1,gro2} for more about the logarithmic inequality. Moreover,
there is a result about curvature that guarantees \eqref{1.1}.
Let us first recall some notions; see, for instance, \cite{bak1,be2,krs}.

For all $u,v,uv\in D(A)$, define the ``square of the length of the gradient"
pointwise by
$$\Gamma(u,v)(x)=\frac 12[A(uv)(x)-u(x)Av(x)-v(x)Au(x)].$$
Further, assume that there exists a dense subspace $\mathscr{S}\subset N^{1,2}(X)$
such that for all $u,v\in \mathscr{S}$, we can define
$$\Gamma_2(u,v)(x)=\frac 12[A(\Gamma(u,v))(x)-\Gamma(u,Av)(x)-\Gamma(v,Au)(x)].$$

The diffusion semigroup is said to have curvature greater or equal to some
$\kappa\in\rr$, if for all $u\in \mathscr{S}$ and $x\in X$,
\begin{equation}\label{4.1}
\Gamma_2(u,u)\ge \kappa \Gamma(u,u).
\end{equation}
The following result is part of \cite[Proposition 2.1]{bak1}.

\begin{prop}\label{p4.1}
Assume that the subspace $\mathscr{S}$ is as above, and that the diffusion
semigroup has curvature greater or equal to some $\kappa\in\rr$.
Then, for every $u\in N^{1,2}(X)$, each $t>0$ and for almost every $x_0\in X$,
it holds
\begin{eqnarray}\label{4.2}
  \int_X(u(x)-T_tu(x))^2p(t,x_0,x)\,d\mu(x)\le \frac{1-e^{-2\kappa t}}{\kappa}
  \int_X|Du(x)|^2p(t,x_0,x)\,d\mu(x),
\end{eqnarray}
when $\kappa=0$, $\frac{1-e^{-2\kappa t}}{\kappa}$ is replaced by $2t$.
Moreover, if inequality \eqref{4.2} holds true for every
$u\in N^{1,2}(X)$ and almost every $x_0\in X$, then \eqref{4.1} holds true for
all functions in some dense subclass $\mathscr{S}\subset N^{1,2}(X)$ as well.
\end{prop}

Since \eqref{4.2} implies \eqref{1.1}, by Proposition \ref{p4.1},
we see that \eqref{1.1} holds when the curvature of the diffusion
semigroup is bounded from below. It is well known that Riemannian manifolds
with Ricci curvature bounded from below satisfies \eqref{4.1},
where the generator $A$ is the Laplace-Beltrami operator;
see, for example, \cite{bak1,chy}.

Another example given in \cite{krs} is the Euclidean spaces with smooth
Ahlfors regular weights. Let $w\in C^2(\rn)$ be an Ahlfors regular weight.
It was shown that if $\frac{1}{w^2}(|\nabla w|^2-w\triangle w)\ge \kappa$,
then for all $u\in C^\fz_0(\rn)$, $\Gamma_2(u,u)\ge \kappa \Gamma(u,u).$
Here, we want to give another explanation which shows
that $\Gamma_2(u,u)\ge 0$ whenever $w$ is a positive $C^2(\rn)$ function.
For every $u\in C_0^\fz(\rn)$, we have $\Gamma(u,u)=|\nabla u|^2$, and
\begin{eqnarray*}
&&\int_\rn \Gamma_2(u,u)(x)w(x)\,dx\\
&&\hs=\int_\rn \frac12[A(|\nabla u|^2)(x)
-2\nabla u(x)\cdot \nabla (Au)(x)]w(x)\,dx\\
&&\hs=-\int_\rn \nabla u(x)\cdot \nabla (\triangle u)(x)w(x)\,dx
-\int_\rn \nabla u(x)\cdot \nabla \lf(\frac{\nabla u \cdot \nabla w}{w}\r)(x)w(x)\,dx\\
&&\hs= \int_\rn (\triangle u(x))^2w(x)\,dx+ \int_\rn
\triangle u(x)\nabla u(x)\cdot \nabla w(x)\,dx\\
&&\hs\hs+ \int_\rn \triangle u(x)\nabla u(x)\cdot \nabla w(x)\,dx+
\int_\rn \frac{|\nabla u(x) \cdot \nabla w(x)|^2}{w(x)}\,dx\\
&&\hs= \int_\rn\lf(\triangle u(x)\sqrt {w(x)}+
\frac{\nabla u(x) \cdot \nabla w(x)}{\sqrt {w(x)}}\r)^2\,dx\ge 0.
\end{eqnarray*}
Thus \eqref{4.1} always holds with $\kappa=0$, whenever
$w\in C^2(\rn)$ is positive. Notice here, the condition
that $w\in C^2(\rn)$ is positive implies that $w$ is a locally
Ahlfors-regular weight.

\subsection*{Acknowledgment}
\hskip\parindent  The author is grateful to his supervisor Professor
Pekka Koskela for posing the problem and many kind suggestions.
He also wishes to express deeply thanks to Kai Rajala, Guo Zhang and
Yuan Zhou for many helpful discussions. Last but not least,
the author would also like to thank the referee for his many valuable remarks
which made this article more readable.

\noindent Renjin Jiang\\
Department of Mathematics and Statistics\\
University of Jyv\"{a}skyl\"{a}\\
P.O. Box 35 (MaD)\\
FI-40014\\
Finland

\medskip

\noindent{\it E-mail address}: \texttt{renjin.r.jiang@jyu.fi}

\begin{thebibliography}{999}

\bibitem{bak1} D. Bakry, On Sobolev and Logarithmic Inequalities for Markov Semigroups, New Trends in
Stochastic Analysis (Charingworth, 1994), World Scientific Publishing, River Edge, NJ, 1997,
pp. 43-75.

\vspace{-0.3cm}
\bibitem{be2} D. Bakry, M. Emery, Diffusions hypercontractives, Seminaire de probabilities,
Vol. XIX, 1983/84, pp. 177-206.

\vspace{-0.3cm}
\bibitem{bm93} M. Biroli, U. Mosco, Sobolev inequalities for
Dirichlet forms on homogeneous spaces.  Boundary value problems
for partial differential equations and applications,
RMA Res. Notes Appl. Math., 29, Masson, Paris, (1993) 305-311.

\vspace{-0.3cm}
\bibitem{bm} M. Biroli, U. Mosco, A Saint-Venant type principle for
Dirichlet forms on discontinuous media, Ann. Mat. Pura Appl. 169 (1995) 125-181.

\vspace{-0.3cm}
\bibitem{ck} L.A. Caffarelli, C.E. Kenig, Gradient estimates for
variable coefficient parabolic equations and singular perturbation
problems, Amer. J. Math. 120 (1998) 391-439.

\vspace{-0.3cm}
\bibitem{ch} J. Cheeger, Differentiability of Lipschitz functions on
metric measure spaces, Geom. Funct. Anal. 9 (1999) 428-517.

\vspace{-0.3cm}
\bibitem{chy} S.Y. Cheng, S.T. Yau, Differential equations on Riemannian
manifolds and their geometric applications, Comm. Pure Appl. Math. 28 (3) (1975) 333-354.

\vspace{-0.3cm}
\bibitem{fot} M. Fukushima, Y. Oshima, M. Takeda, Dirichlet Forms and Symmetric
Markov Processes, in: de Gruyter Studies in Mathematics, Vol. 19,
Walter de Gruyter \& Co., Berlin, 1994.

\vspace{-0.3cm}
\bibitem{gro1} L. Gross, Logarithmic Sobolev inequalities, Amer. J. Math. 97 (1975) 1061-1083.

\vspace{-0.3cm}
\bibitem{gro2} L. Gross, Hypercontractivity over complex manifolds, Acta Math. 182 (1999) 159-206.

\vspace{-0.3cm}
\bibitem{hak95} P. Haj{\l}asz, P. Koskela, Sobolev meets Poincar\'e,
C. R. Acad. Sci. Paris S\'er. I Math. 320 (10) (1995) 1211-1215.

\vspace{-0.3cm}
\bibitem{hak} P. Haj{\l}asz, P. Koskela, Sobolev met Poincar\'e,
Mem. Amer. Math. Soc. 145 (688) (2000).

\vspace{-0.3cm}
\bibitem{hek} J. Heinonen, P. Koskela, Quasiconformal maps in metric spaces
with controlled geometry,  Acta Math.  181  (1998) 1-61.

\vspace{-0.3cm}
\bibitem{kkm} T. Kilpel\"ainen, J. Kinnunen, O. Martio,
Sobolev spaces with zero boundary values on metric spaces,
Potential Anal. 12 (3) (2000) 233-247.

\vspace{-0.3cm}
\bibitem{krs} P. Koskela, K. Rajala, N. Shanmugalingam, Lipschitz
continuity of Cheeger-harmonic functions in metric measure spaces,
J. Funct. Anal. 202 (2003) 147-173.

\vspace{-0.3cm}
\bibitem{sal} L. Saloff-Coste, A note on Poincar\'e, Sobolev,
and Harnack inequalities, Internat. Math. Res. Notices (2) (1992) 27-38.

\vspace{-0.3cm}
\bibitem{sem} S. Semmes, in: M. Gromov (Ed.), Metric Structures
for Riemannian and Non-Riemannian Spaces, Appendix, Progress
in Mathematics, Vol. 152, Birkh\"auser Boston, Inc., Boston, MA, 1999.

\vspace{-0.3cm}
\bibitem{sh} N. Shanmugalingam, Newtonian spaces: an extension of Sobolev
spaces to metric measure spaces,  Rev. Mat. Iberoamericana 16 (2000)
243-279.

\vspace{-0.3cm}
\bibitem{st1} K.T. Sturm, Analysis on local Dirichlet spaces. I. Recurrence,
conservativeness and $L^p$-Liouville properties, J. Reine Angew. Math.
456 (1994) 173-196.

%\vspace{-0.3cm}
%\bibitem{st2} K.T. Sturm, Analysis on local Dirichlet spaces. II. Upper
%Gaussian estimates for the fundamental solutions of parabolic
%equations, Osaka J. Math. 32 (2) (1995) 275-312.

\vspace{-0.3cm}
\bibitem{st3} K.T. Sturm, Analysis on local Dirichlet spaces. III. The
parabolic Harnack inequality, J. Math. Pures Appl. (9) 75 (3) (1996)
273-297.

\end{thebibliography}
\end{document}